\documentclass[a4paper,10pt]{amsart}

\usepackage[english]{babel}
\usepackage{amsmath,amssymb,amsthm}
\usepackage{hyperref}
\usepackage[centering]{geometry}
\usepackage{color}

\newcommand{\metric}[2]{\ensuremath{\langle #1, #2\rangle}} 
\newcommand{\HH}{\ensuremath{\mathbb{H}^2\times \mathbb{H}^2}} 
\renewcommand{\epsilon}{\varepsilon} 

\newcommand{\R}{\ensuremath{\mathbb{R}}}
\newcommand{\C}{\ensuremath{\mathbb{C}}}
\renewcommand{\H}{\ensuremath{\mathbb{H}}}

\newtheorem{theorem}{Theorem}[section] 
\newtheorem*{theorem*}{Theorem} 
\newtheorem{lemma}[theorem]{Lemma}
\newtheorem{proposition}[theorem]{Proposition}

\theoremstyle{definition}

\newtheorem{example}[theorem]{Example}

\newtheorem*{remark}{Remark}

\numberwithin{equation}{section}

\title[Lagrangian surfaces in $\HH$]{Lagrangian surfaces in $\HH$ }

\author{Dong Gao \and Joeri Van der Veken \and Anne Wijffels \and Botong Xu}

\address{D.~Gao, Department of Mathematics, School of Science, Beijing University of Civil Engineering and Architecture, Beijing 102616, P.R. China}
\email{gaodong@bucea.edu.cn}
\address{J. Van der Veken \and A. Wijffels, KU\ Leuven, Department of Mathematics, Celestijnenlaan 200B -- Box 2400, 3001 Leuven, Belgium}
\email{joeri.vanderveken@kuleuven.be, anne.wijffels@kuleuven.be}
\address{B.~Xu, Department of Mathematical Sciences, Tsinghua University, Beijing, 100084, P.R. China}
\email{xbt17@mails.tsinghua.edu.cn}
\thanks{This research is supported by collaboration project G0F2319N of the Research Foundation—Flanders (FWO) and the National Natural Science Foundation of China (NSFC). J. Van der Veken and A. Wijffels are supported by project 3E160361 of the KU Leuven Research Fund and J. Van der Veken is supported by the Research Foundation—Flanders (FWO) and the Fonds de la Recherche Scientifique (FNRS) under EOS Project G0H4518N}

\subjclass[2010]{Primary 53C42, 53D12; Secondary 53B25}
\keywords{Lagrangian surface, minimal surface, Riemannian product space}

\date{}

\begin{document}

\begin{abstract}
The Riemannian product of two hyperbolic planes of constant Gaussian curvature $-1$ has a natural K\"ahler structure. In fact, it can be identified with the complex hyperbolic quadric of complex dimension two. In this paper we study Lagrangian surfaces in this manifold. We present several examples and classify the totally umbilical and totally geodesic Lagrangian surfaces, the Lagrangian surfaces with parallel second fundamental form, the minimal Lagrangian surfaces with constant Gaussian curvature and the complete minimal Lagrangian surfaces satisfying a bounding condition on an important function that can be defined on any Lagrangian surface in this particular ambient space.
\end{abstract}

\maketitle

\section{Introduction}

In \cite{CU}, I. Castro and F. Urbano studied Lagrangian surfaces in the Riemannian product of two unit spheres, $\mathbb{S}^2\times \mathbb{S}^2$. Moreover, they showed how this ambient space is isometric to the two-dimensional complex quadric $Q^2$. 

In this paper, we study Lagrangian surfaces of the closely related manifold $\HH$, the Riemannian product of two hyperbolic planes of curvature $-1$. Just like $\mathbb{S}^2\times \mathbb{S}^2$, it is a K\"ahler-Einstein space. Moreover, it is homothetic to the two-dimensional complex hyperbolic quadric $Q^{2\ast}$. The latter manifold is the target space of Gauss maps of spacelike surfaces in the three-dimensional anti-de Sitter space $\H^3_1(-1)$ and the image of such a Gauss map is a Lagrangian surface, see \cite{VdVW}. This provides us immediately with a large family of examples of Lagrangian surfaces in $\HH$. 

In Section 3 of the paper, we define a function $\Gamma: \Sigma \to \R$ on a Lagrangian surface $\Sigma$ in $\HH$ and we prove that $0 \leq \Gamma^2 \leq 1/4$ and that $\Gamma^2 \equiv 0$ if and only if $\Sigma$ is a \textit{product of curves} 
$$ \Phi: I_1 \times I_2 \subseteq \R^2 \to \HH: (s_1,s_2) \mapsto (\beta_1(s_1),\beta_2(s_2)) $$
and $\Gamma^2 \equiv 1/4$ if and only if $\Sigma$ is the image of the \textit{diagonal immersion}
$$ \Phi: \H^2(-1/2) \to \HH: x \mapsto \frac{1}{\sqrt 2}(x,x), $$
which has constant Gaussian curvature $-1/2$ (Theorem \ref{theo:characterization_by_Gamma}). In Section 4, we classify all Lagrangian surfaces in $\HH$ with parallel second fundamental form (Theorem \ref{theo:parallel_sff}) and all the totally umbilical and totally geodesic Lagrangian surfaces in $\HH$ (Theorem \ref{theo:totally_geodesic} and the remark thereafter). Both theorems show the importance of the above mentioned examples: the product of curves and the diagonal immersion. Finally, in Section 5, we prove two results on minimal Lagrangian surfaces in $\HH$. In particular, we give a full classification of minimal Lagrangian surfaces in $\HH$ with constant Gaussian curvature and of complete minimal Lagrangian surfaces in $\HH$ for which the function $\Gamma^2$ is bounded away from $1/4$.

\section{Models for $\HH$}

For integers $k$ and $n$, satisfying $0 \leq k \leq n$, we denote by $\R^n_k$ the pseudo-Euclidean space of dimension $n$ and index $k$. More specifically, $\R^n_k$ is $\R^n$ with the metric
$$ \metric{(a_1, \ldots, a_n)}{(b_1, \ldots, b_n)} = -a_1b_1 - \ldots -a_kb_k + a_{k+1}b_{k+1} + \ldots + a_nb_n. $$

\subsection{$\H^2(c)\times\H^2(c)$ as a submanifold of $\R^6_2$} The three-dimensional Minkowski space $\R^3_1$ carries a Lorentzian cross product $\boxtimes$ (see e.g. \cite{DM} and \cite{K}), defined by
\begin{equation*}
(a_1,a_2,a_3) \boxtimes (b_1,b_2,b_3) = (a_3b_2-a_2b_3, \, a_3b_1-a_1b_3, \, a_1b_2-a_2b_1).
\end{equation*} 
Note that the only difference with the Euclidean cross product is the sign of the first component. The Lorentzian cross product is a bilinear operation satisfying the familiar properties of a cross product with respect to the Lorentzian metric $\metric{\cdot\,}{\cdot}$, namely,
\begin{align}\label{eq:properties_cross_product}
& a \boxtimes b = - b \boxtimes a, \nonumber \\
& \metric{a}{a \boxtimes b} = \metric{b}{a \boxtimes b} = 0, \\
& \metric{a \boxtimes b}{a \boxtimes b} = -\metric{a}{a}\metric{b}{b}+ \metric{a}{b}^2 \nonumber
\end{align}
for all $a,b\in\R^3_1$. 

The hyperbolic plane of curvature $c<0$ can be defined as the following subset of $\R^3_1$: 
\begin{equation*}
\H^2(c) = \{(x_1,x_2,x_3) \in \R^3_1 \ | -x_1^2 + x_2^2 + x_3^3 = 1/c \mbox{ and } x_1>0 \},
\end{equation*}
where $c$ is a negative real constant. The induced metric from $\metric{\cdot\,}{\cdot}$ turns $\H^2(c)$ into a complete simply connected Riemannian surface of constant Gaussian curvature $c$. The standard complex structure and its associated K\"ahler two-form on $\H^2(c)$ are defined by 
\begin{equation} \label{eq:J_omega_H2}
\begin{aligned}
& (J_{\H^2(c)})_x v = \sqrt{-c} \, x \boxtimes v, \\
& (\omega_{\H^2(c)})_x (v,w) = \metric{(J_{\H^2(c)})_x v}{w}
\end{aligned}
\end{equation}
for all $x \in \H^2(c)$ and all $v,w \in T_x\H^2(c)$.

%

With this model for $\H^2(c)$, the product manifold $\H^2(c)\times\H^2(c)$ naturally becomes a subset of $\R_1^3 \times \R_1^3 \cong \R_2^6$. The induced metric on $\H^2(c)\times\H^2(c)$ is the Riemannian product of the above defined metric on each factor, which we will also denote by $\langle \cdot \, , \cdot \rangle$. We define a complex structure on $\H^2(c)\times\H^2(c)$ by $J =(J_{\H^2(c)}, -J_{\H^2(c)})$, that is, 
\begin{equation} \label{eq:defJ}
J_{(x_1,x_2)} (v_1,v_2) = ((J_{\H^2(c)})_{x_1} v_1, -(J_{\H^2(c)})_{x_2} v_2) =
\sqrt{-c}\left(x_1 \boxtimes v_1, -x_2 \boxtimes v_2
\right)
\end{equation}  
for all $(x_1,x_2) \in \H^2(c)\times\H^2(c)$ and all $(v_1, v_2)\in T_{(x_1,x_2)}(\H^2(c)\times\H^2(c)) \cong T_{x_1}\H^2(c) \oplus T_{x_2}\H^2(c)$. The associated K\"ahler two-form $\omega$ of $\H^2(c)\times\H^2(c)$ is given by 
\begin{equation}
\omega = \langle J \cdot \, , \cdot \rangle = \pi_1^* \omega_{\H^2(c)} - \pi_2^* \omega_{\H^2(c)}, 
\end{equation}
where $\pi_1: \H^2(c)\times\H^2(c) \to \H^2(c): (x_1,x_2) \mapsto x_1$ and $\pi_2: \H^2(c)\times\H^2(c) \to \H^2(c): (x_1,x_2) \mapsto x_2$ are the projections from $\H^2(c)\times\H^2(c)$ to its two factors. The isometry group of $\H^2(c)\times\H^2(c)$ is
$$ \mathrm{Iso}(\H^2(c)\times\H^2(c)) = \left\{ \left. \left( \begin{array}{cc}A_1&0\\ 0&A_2\end{array} \right) , \left( \begin{array}{cc}0&B_1\\ B_2&0\end{array} \right) \ \right| \ A_1,A_2,B_1,B_2 \in \mathrm{O}^+(1,2)\right\}, $$
where $\mathrm{O}^+(1,2)$ denotes the ortochronous Lorentz group. The holomorphic isometries of $\H^2(c)\times\H^2(c)$, i.e., the elements of $\mathrm{Iso}(\H^2(c)\times\H^2(c))$ which preserve the complex structure $J$, are those elements for which $\det A_1 = \det A_2 = 1$ or $\det B_1 = \det B_2 = -1$. The anti-holomorphic isometries of $\H^2(c)\times\H^2(c)$, i.e., the elements of $\mathrm{Iso}(\H^2(c)\times\H^2(c))$ which take $J$ to $-J$, are those elements for which $\det A_1 = \det A_2 = -1$ or $\det B_1 = \det B_2 = 1$. In particular, there are isometries of $\H^2(c)\times\H^2(c)$ that are neither holomorphic or anti-holomorphic.

Finally, we remark that the curvature tensor of $\H^2(c)\times\H^2(c)$ is given by
\begin{equation} \label{eq:curvature_tensor_HH}
\begin{aligned}
\langle R^{\H^2(c)\times\H^2(c)}(X,Y)Z,W \rangle = 
& \ c(\langle \pi_{1\ast}X,\pi_{1\ast}Z \rangle \langle \pi_{1\ast}Y,\pi_{1\ast}W \rangle - \langle \pi_{1\ast}Y,\pi_{1\ast}Z \rangle \langle \pi_{1\ast}X,\pi_{1\ast}W \rangle \\
& \ + \langle \pi_{2\ast}X,\pi_{2\ast}Z \rangle \langle \pi_{2\ast}Y,\pi_{2\ast}W \rangle - \langle \pi_{2\ast}Y,\pi_{2\ast}Z \rangle \langle \pi_{2\ast}X,\pi_{2\ast}W \rangle)
\end{aligned}
\end{equation}
for all vector fields $X$, $Y$, $Z$ and $W$ tangent to $\H^2(c)\times\H^2(c)$.

\subsection{$\H^2(-4)\times\H^2(-4)$ as the complex hyperbolic quadric $Q^{2\ast}$} 
\label{subsec:Q*2}

The complex hyperbolic quadric $Q^{n\ast}$, see for example \cite{BS} and \cite{VdVW}, is a homogeneous K\"ahler-Einstein manifold, which exists for any complex dimension $n$. In this subsection, we will briefly introduce it for $n=2$ and show that, in this case, it is holomorphically isometric to $\mathbb H^2(-4) \times \H^2(-4)$. 

One way to introduce $Q^{2\ast}$ is as a connected component of the Grassmannian of negative definite, oriented two-planes in the pseudo-Euclidean space $\R^4_2$. It is clear that this Grassmannian can be identified with the symmetric space $\mathrm{SO}(2,2) / (\mathrm{SO}(2) \times \mathrm{SO}(2))$ (see also the remark below). The Lie group $\mathrm{SO}(2,2)$ has two connected components and one defines
\begin{equation} \label{eq:Q2*1}
Q^{2\ast} = \frac{\mathrm{SO}(2,2)^0}{\mathrm{SO}(2) \times \mathrm{SO}(2)},
\end{equation}
where the superscript $0$ denotes the connected component of the identity matrix. 

\begin{remark}
Recall that 
$$ \mathrm{SO}(2,2) = \left\{ A \in \R^{4 \times 4} \ \left| \ A \left( \! \begin{array}{cccc} -1&0&0&0 \\ 0&-1&0&0 \\ 0&0&1&0 \\ 0&0&0&1 \end{array} \! \right) A^T = \left( \! \begin{array}{cccc} -1&0&0&0 \\ 0&-1&0&0 \\ 0&0&1&0 \\ 0&0&0&1 \end{array} \! \right), \ \det A = 1 \right. \right\}, $$ 
where $A^T$ denotes the transpose of $A$. The columns of an element of $\mathrm{SO}(2,2)$ form a pseudo-orthonormal basis for $\R^4_2$. More precisely, if the columns from left to right are labelled $u_1$, $u_2$, $u_3$ and $u_4$, then these vectors are orthogonal and satisfy $\langle u_1,u_1 \rangle = \langle u_2,u_2 \rangle = -1$ and $\langle u_3,u_3 \rangle = \langle u_4,u_4 \rangle = 1$. The above mentioned identification between the Grassmannian of negative definite, oriented two-planes in $\R^4_2$ and $\mathrm{SO}(2,2) / (\mathrm{SO}(2) \times \mathrm{SO}(2))$ is then obtained by taking a positively oriented pseudo-orthonormal basis $(u_1,u_2)$ in such a two-plane and an ordered pseudo-orthonormal basis $(u_3,u_4)$ in its orthogonal complement, such that $(u_1,u_2,u_3,u_4)$ is positively oriented with respect to the standard orientation of $\R^4_2$. The matrix with columns $u_1$, $u_2$, $u_3$ and $u_4$ belongs to $\mathrm{SO}(2,2)$ and since both $(u_1,u_2)$ and $(u_3,u_4)$ are only determined up to a rotation, we indeed obtain de desired identification. 

If we write an element $A \in \mathrm{SO}(2,2)$ as a block matrix 
$$ A = \left( \begin{array}{cc} A_{11} & A_{12} \\ A_{21} & A_{22} \end{array} \right),  $$
with $A_{11}, A_{12}, A_{21}, A_{22} \in \R^{2 \times 2}$, then $\det A_{11} = \det A_{22}$ and $\det A_{12} = \det A_{21}$. Moreover, the expression $\det A_{11}-\det A_{21}$ is never zero and hence its sign distinguishes the two connected components of $\mathrm{SO}(2,2)$. The component $\mathrm{SO}(2,2)^0$ is characterized by $\det A_{11}-\det A_{21}>0$.
\end{remark}

A second description of $Q^{2\ast}$, which is more suited to introduce the Riemannian metric and the complex structure, is as a complex hypersurface of a complex anti-de Sitter space. Define $\mathbb{C}H^3_1$ as the set of all complex 1-dimensional subspaces of $\mathbb{C}^4_2 \cong \R^8_4$ on which the induced metric is negative definite. If we denote the complex span of a vector $z \in \mathbb C^4_2$ by $[z]$, the map $\pi$ from $\mathbb H^7_3(-1)= \{z \in \mathbb{C}^4_2 \cong \R^8_4 \ | \ \langle z,z \rangle = -1\}$ to $\mathbb{C}H^3_1$, defined by
$$ \pi: \mathbb H^7_3(-1) \to \mathbb{C}H^3_1: z \mapsto [z], $$
is called the Hopf fibration. The pseudo-Riemannian metric $g$ on $\mathbb{C}H^3_1$ is such that $\pi$ becomes a pseudo-Riemannian submersion and the complex structure $J$ is induced by multiplying by~$i$, in the sense that $J(d\pi)v = (d\pi)(iv)$ for all horizontal vectors $v \in T\mathbb H^7_3(-1)$. The complex hyperbolic quadric is then given by
\begin{equation} \label{eq:Q2*2}
Q^{2\ast} = \{[(z_1,z_2,z_3,z_4)]\in\mathbb{C}H^3_1 \ | \ -z_1^2 - z_2^2 +z_3^2 + z_4^2=0 \}^0,
\end{equation}
where the superscript $0$ again denotes a connected component. The metric induced by $g$ is Riemannian and the almost complex structure induced by $J$ makes $Q^{\ast 2}$ a K\"ahler-Einstein manifold. From now on, we assume that $Q^{2\ast}$ is equipped with these induced structures, which we again denote by $g$ and $J$ respectively.

\begin{remark}
The equivalence between \eqref{eq:Q2*1} and \eqref{eq:Q2*2} follows from identifying a negative definite, oriented two-plane with positively oriented pseudo-orthonormal basis $(u_1,u_2)$ in $\R^4_2$ with the point $[u_1+iu_2]$ of $\mathbb C H^3_1$. Note that the connected component of the equivalence class of the identity matrix in \eqref{eq:Q2*1} corresponds to the connected component of $[(1,i,0,0)]$ in \eqref{eq:Q2*2}.
\end{remark}

To construct an isometry between $Q^{2\ast}$ and $\mathbb H^2(-4) \times \mathbb H^2(-4)$, consider the following space of two-vectors: $\bigwedge^2 \R^4_2 = \{v\wedge w \ | \ v,w \in \R^4_2\} $. Equipped with the metric
\begin{equation}\label{eq:metric_2vectors}
\langle\!\langle v_1 \wedge w_1 , v_2 \wedge w_2 \rangle\!\rangle = - \langle v_1,v_2 \rangle \langle w_1,w_2 \rangle + \langle v_1,w_2 \rangle \langle v_2,w_1 \rangle, 
\end{equation}
we obtain that $\bigwedge^2 \R^4_2$ is isometric to $\R^6_2$. Note that the sign of \eqref{eq:metric_2vectors} is different from the usual convention in order to obtain a space containing $\mathbb H^2(-4) \times \mathbb H^2(-4)$ without further identifications. If $u=(u_1,u_2,u_3,u_4)$ is a positively oriented pseudo-orthonormal basis of $\R^4_2$, with $\langle u_1,u_1 \rangle = \langle u_2,u_2 \rangle = -1$ and $\langle u_3,u_3 \rangle = \langle u_4,u_4 \rangle = 1$, the Hodge star operator on $\bigwedge^2 \R^4_2$ is completely determined by $\ast(u_1 \wedge u_2) = u_4 \wedge u_3$, $\ast(u_1 \wedge u_3) = u_4 \wedge u_2$, $\ast(u_1 \wedge u_4) = u_2 \wedge u_3$ and $\ast^2=\mathrm{id}$.
If we define
\begin{align*}
& E_{1\pm}(u) = \frac{1}{\sqrt{2}}(u_1\wedge u_2 \pm u_4\wedge u_3), \\
& E_{2\pm}(u) = \frac{1}{\sqrt{2}}(u_1\wedge u_3 \pm u_4\wedge u_2), \\
& E_{3\pm}(u) = \frac{1}{\sqrt{2}}(u_1\wedge u_4 \pm u_2\wedge u_3),
\end{align*}
then $\{E_{1+}(u),E_{2+}(u),E_{3+}(u)\}$ is a pseudo-orthonormal basis for the space of self-dual two-vectors $\bigwedge^2_+ \R^4_2$, whereas $\{E_{1-}(u),E_{2-}(u),E_{3-}(u)\}$ is a pseudo-orthonormal basis for the space of anti-self-dual two-vectors $\bigwedge^2_- \R^4_2$. Note that the splitting $\bigwedge^2 \R^4_2 = (\bigwedge^2_+ \R^4_2) + (\bigwedge^2_- \R^4_2$) is orthogonal and that both subspaces are isometric to $\R^3_1$, since $\langle\!\langle E_{1\pm}(u),E_{1\pm}(u) \rangle\!\rangle=-1$ and $\langle\!\langle E_{2\pm}(u),E_{2\pm}(u) \rangle\!\rangle = \langle\!\langle E_{3\pm}(u),E_{3\pm}(u) \rangle\!\rangle = 1$. In particular, both spaces contain hyperbolic planes. For example, if $e=(e_1,e_2,e_3,e_4)$ is the standard basis of $\R^4_2$, then 
$$ \mathbb H^2_{\pm}(c) = \left\{ \left. \mbox{$x_1 E_{1\pm}(e) + x_2 E_{2\pm}(e) + x_3 E_{3\pm}(e) \in \bigwedge^2_{\pm} \R^4_2$} \ \right| \, -x_1^2+x_2^2+x_3^2 = 1/c \mbox{ and } x_1 > 0 \right\} $$
are complex hyperbolic planes of curvature $c<0$. We will omit the subscripts in the notation, so, for the rest of this section, $\mathbb H^2(c) \subseteq \bigwedge^2_+ \R^4_2$ and $\mathbb H^2(c) \subseteq \bigwedge^2_- \R^4_2$ will denote the above subsets.

We can now explicitly give the holomorphic isometry between $Q^{2\ast}$ and $\mathbb H^2(-4) \times \mathbb H^2(-4)$.
\begin{proposition}
For a negative definite oriented two-plane $P$ in $\R^4_2$, we denote by $u_P=(u_1,u_2,u_3,u_4)$ a positively oriented pseudo-orthonormal basis of $\R^4_2$ such that $(u_1,u_2)$ is a positively oriented basis for $P$. Then
$$ \phi: Q^{\ast 2} \to \mbox{$(\bigwedge^2_+ \R^4_2) \oplus (\bigwedge^2_- \R^4_2)$} : P \mapsto \left( \frac 12 E_{1+}(u_P), \frac 12 E_{1-}(u_P) \right) $$
is a well-defined holomorphic isometry between $Q^{\ast 2}$ and $\mathbb H^2(-4) \times \mathbb H^2(-4)$.
\end{proposition}

\begin{proof} 
For a given $P \in Q^{2\ast}$, the basis $u_P$ is determined up to rotations of $(u_1,u_2)$ and $(u_3,u_4)$. Since such rotations do not change $E_{1+}(u_P)$ and $E_{1-}(u_P)$, the map $\phi$ is well-defined.

We now show that $\phi$ is a bijection between $Q^{\ast 2}$ and $\mathbb H^2(-4) \times \mathbb H^2(-4)$. Let $e=(e_1,e_2,e_3,e_4)$ be the standard basis of $\R^4_2$ and let $u_P$ be a basis associated to some $P \in Q^{2\ast}$. After a suitable rotation of $(u_1,u_2)$, we may assume that the projections of $u_1$ and $u_2$ to $\mathrm{span}\{e_1,e_2\}$ and to $\mathrm{span}\{e_3,e_4\}$ are orthogonal. Similarly, after a suitable rotation of $(u_3,u_4)$, we may assume that the projections of $u_3$ and $u_4$ to $\mathrm{span}\{e_1,e_2\}$ and to $\mathrm{span}\{e_3,e_4\}$ are orthogonal. More specifically, for every $P \in Q^{2\ast}$ there exists a basis $u_P$ of the form
\begin{align*}
& u_1 = \cosh A (\cos\alpha \, e_1 + \sin\alpha \, e_2) + \sinh A (\cos\beta \, e_3 + \sin\beta \, e_4), \\
& u_2 = \cosh B (-\sin\alpha \, e_1 + \cos\alpha \, e_2) + \sinh B (-\sin\beta \, e_3 + \cos\beta \, e_4), \\
& u_3 = \sinh A (\cos\alpha \, e_1 + \sin\alpha \, e_2) + \cosh A (\cos\beta \, e_3 + \sin\beta \, e_4), \\
& u_4 = \sinh B (-\sin\alpha \, e_1 + \cos\alpha \, e_2) + \cosh B (-\sin\beta \, e_3 + \cos\beta \, e_4).
\end{align*}
The numbers $A$ and $B$ are uniquely determined, whereas the angles $\alpha$ and $\beta$ are determined up to an integer multiple of $\pi$. A straightforward computation yields
\begin{align*}
& E_{1+}(u_P) = \cosh(A-B)E_{1+}(e) + \sinh(A-B)(\sin(\alpha+\beta)E_{2+}(e)-\cos(\alpha+\beta)E_{3+}(e)), \\
& E_{1-}(u_P) = \cosh(A+B)E_{1-}(e) + \sinh(A+B)(\sin(\alpha-\beta)E_{2-}(e)+\cos(\alpha-\beta)E_{3+}(e)),
\end{align*}
from which it is clear that $\phi$ will attain every point of $\mathbb H^2(-4) \times \mathbb H^2(-4)$ exactly once. From the formulas, we see that $\phi$ is not only a bijection between $Q^{\ast 2}$ and $\mathbb H^2(-4) \times \mathbb H^2(-4)$, but even a diffeomorphism.

To see that $\phi$ preserves the metric and the complex structure, we use the model \eqref{eq:Q2*2} for $Q^{2\ast}$. In this model, a negative definite oriented two-plane $P \subseteq \R^4_2$, with $u_P=(u_1,u_2,u_3,u_4)$, corresponds to the point $[u_1+iu_2]$. Moreover, the tangent space to $Q^{2\ast}$ at $[u_1+iu_2]$ is spanned by $(d\pi)_{\frac{1}{\sqrt 2}(u_1+iu_2)}(u_3)$, $(d\pi)_{\frac{1}{\sqrt 2}(u_1+iu_2)}(u_4)$, $(d\pi)_{\frac{1}{\sqrt 2}(u_1+iu_2)}(iu_3)$ and $(d\pi)_{\frac{1}{\sqrt 2}(u_1+iu_2)}(iu_4)$. Note that these vectors are orthonormal, since $\pi$ is a Riemannian submersion and $u_3$, $u_4$, $iu_3$ and $iu_4$ are horizontal orthonormal tangent vectors to $\H^7_3(-1) \subseteq \C^4_2 \cong \R^8_4$. A straightforward computation shows
\begin{equation} \label{eq:dphi}
\begin{aligned}
&(d\phi)_{[u_1+iu_2]}(d\pi)_{\frac{1}{\sqrt 2}(u_1+iu_2)}(u_3) = \frac{1}{\sqrt 2}(-E_{3+}(u_P),E_{3-}(u_P)),\\
&(d\phi)_{[u_1+iu_2]}(d\pi)_{\frac{1}{\sqrt 2}(u_1+iu_2)}(u_4) = \frac{1}{\sqrt 2}(E_{2+}(u_P),-E_{2-}(u_P)),\\
&(d\phi)_{[u_1+iu_2]}(d\pi)_{\frac{1}{\sqrt 2}(u_1+iu_2)}(iu_3) = \frac{1}{\sqrt 2}(E_{2+}(u_P),E_{2-}(u_P)),\\
&(d\phi)_{[u_1+iu_2]}(d\pi)_{\frac{1}{\sqrt 2}(u_1+iu_2)}(iu_4) = \frac{1}{\sqrt 2}(E_{3+}(u_P),E_{3-}(u_P)).
\end{aligned}
\end{equation}
Since the vectors on the right hand sides of \eqref{eq:dphi} are orthonormal, we conclude that $\phi$ is indeed an isometry between $Q^{2\ast}$ and $\mathbb H^2(-4) \times \mathbb H^2(-4)$. Finally, to show that $\phi$ preserves the complex structure, we see from \eqref{eq:dphi} that it suffices to verify $J(-E_{3+}(u_P),E_{3-}(u))=(E_{2+}(u_P),E_{2-}(u_P))$ and $J(E_{2+}(u_P),-E_{2-}(u_P))=(E_{3+}(u_P),E_{3-}(u_P))$, where $J$ is defined by \eqref{eq:defJ}. This follows from the observation that the basis $(E_{1\pm}(u_P),E_{2\pm}(u_P),E_{3\pm}(u_P))$ behaves as the standard basis $((1,0,0),(0,1,0),(0,0,1))$ with respect to $\boxtimes$.
\end{proof}

\section{Lagrangian surfaces in $\HH$}

From now on, we will abbreviate $\H^2(-1)$ and $\H^2(-1)\times\H^2(-1)$ as $\H^2$ and $\HH$ respectively. 

A two-plane $P$ in the tangent space of a four-dimensional manifold, equipped with a complex structure $J$, is said to be \textit{Lagrangian} with respect to $J$ if $J$ maps $P$ to its orthogonal complement. Equivalently, the two-plane $P$ is Lagrangian if $\omega|_P=0$, where $\omega$ is the K\"ahler two-form associated to $J$. The following lemma can be proven by a straightforward computation.

\begin{lemma} \label{lem:Lagrangian_plane}
Let $P \subseteq T_{(x_1,x_2)}(\HH)$ be a two-plane and assume that $\{(u_1,u_2),(v_1,v_2)\}$ is an orthonormal basis for $P$. Then, the following assertions are equivalent:
\begin{itemize}
\item[(a)] $P$ is Lagrangian with respect to $J=(J_{\H^2},-J_{\H^2})$ or $J'=(J_{\H^2},J_{\H^2})$;
\item[(b)] $\|u_1\| = \|v_2\|$ and $\|u_2\| = \|v_1\|$;
\item[(c)] $\|u_1\|^2 + \|v_1\|^2 = \|u_2\|^2 + \|v_2\|^2 = 1$.
\end{itemize}
\end{lemma}

In the rest of the paper, we will interpret ``Lagrangian in $\HH$'' as ``Lagrangian with respect to $J=(J_{\H^2},-J_{\H^2})$''. An immersion $\Phi: \Sigma \to \HH$ of a surface $\Sigma$ into $\HH$ is Lagrangian if all the tangent planes to the surface are Lagrangian. Denoting $\phi_1 = \pi_1 \circ \Phi$ and $\phi_2 = \pi_2 \circ \Phi$, it is easy to see that $\Phi=(\phi_1,\phi_2): \Sigma \to \HH$ is Lagrangian if and only if
\begin{equation} \label{eq:pullback_omega}
\phi_1^* \omega_{\H^2} = \phi_2^*\omega_{\H^2}.
\end{equation}
We give some important examples. 

\begin{example} \label{ex:product_of_curves}
Let $I_1,I_2 \subseteq \R$ be open intervals and let $\beta_1:I_1\to\H^2$ and $\beta_2:I_2\to\H^2$ be curves parametrized by arc length. Then
\begin{equation} \label{eq:product_of_curves}
\Phi: I_1 \times I_2 \subseteq \R^2 \to \HH: (s_1,s_2) \mapsto (\beta_1(s_1),\beta_2(s_2)) 
\end{equation}
is an isometric immersion of an open part of the Euclidean plane into $\HH$, in particular, the image of $\Phi$ is a flat surface in $\HH$. Moreover, it is easy to see that $\Phi$ is a Lagrangian immersion. Let $N_1 = J_{\H^2}\beta_1' = \beta_1 \boxtimes \beta_1'$ and $N_2 = -J_{\H^2}\beta_2' = -\beta_2 \boxtimes \beta_2'$ be unit normal vector fields along the component curves, tangent to $\H^2$, and put $e_1=(\beta_1',0)$ and $e_2=(0,\beta_2')$. Then the second fundamental form of $\Phi$ with respect to the orthonormal basis $\{e_1,e_2\}$ is given by
\begin{equation} \label{eq:sff_product}
h(e_1,e_1) = (\kappa_1N_1,0), \qquad h(e_1,e_2) = 0, \qquad h(e_2,e_2) = (0,\kappa_2N_2), 
\end{equation}
where $\kappa_1$ and $\kappa_2$ are the geodesic curvatures of $\beta_1$ and $\beta_2$ with respect to the chosen orientations on the factors of $\HH$.
\end{example}

\begin{example} \label{ex:diagonal}
Consider the map
\begin{equation} \label{eq:diagonal}
\Phi: \H^2(-1/2) \to \HH: x \mapsto \frac{1}{\sqrt 2}(x,x)
\end{equation}
from the hyperbolic plane of curvature $-1/2$ to $\HH$. It is easy to see that $\Phi$ is a Lagrangian isometric immersion. We will refer to this immersion as the \textit{diagonal}. The tangent plane to the image of $\Phi$ at a point $\Phi(x)$ is $\{(v,v) \ | \ v \in \R^3_1, \ \langle v,x \rangle = 0\} \subset T_{\Phi(x)}(\HH)$ and the normal plane to the surface at that point is $\{(v,-v) \ | \ v \in \R^3_1, \ \langle v,x \rangle = 0\} \subset T_{\Phi(x)}(\HH)$. From this, it is easy to see that $\Phi$ is totally geodesic. 
\end{example}

\begin{example} \label{ex:graph}
Consider a differentiable map $F: U \subseteq \H^2 \to \H^2$ from an open part $U$ of $\H^2$ to $\H^2$. The graph of $F$ is the immersion
\begin{equation} \label{eq:graph}
\Phi: U \subseteq \H^2 \to \HH: x \mapsto (x,F(x)).
\end{equation}
The metric one needs to put on $U$ to make $\Phi$ into an isometric immersion depends on $F$. However, it follows immediately from \eqref{eq:pullback_omega} that $\Phi$ is a Lagrangian immersion if and only if $F^{\ast}\omega_{\H^2}=\omega_{\H^2}$, i.e., if and only if $F$ is area preserving and orientation preserving.
\end{example}

\begin{example} \label{ex:Gauss_map}
Let $a: \Sigma \to \H^3_1(-1) \subseteq \R^4_2$ be a spacelike immersion of a surface in the three-dimensional anti-de Sitter space of curvature $-1$. A unit normal along the immersion will be timelike and hence we can also consider it as a map $b: \Sigma \to \H^3_1(-1) \subseteq \R^4_2$. The Gauss map
$$ \Phi: \Sigma \to Q^{2\ast}: p \mapsto [a(p)+ib(p)] $$
of $a$ is a Lagrangian immersion of $\Sigma$ into $Q^{\ast 2}$, where we have used the definition \eqref{eq:Q2*2} for $Q^{2\ast}$. By the identification explained in section \ref{subsec:Q*2}, $\Phi$ can be seen as a Lagrangian immersion into $\H^2(-4)\times\H^2(-4)$, which is homothetic to $\HH$.
\end{example}

Let $\Phi: \Sigma \to \HH$ be a Lagrangian immersion of a surface $\Sigma$ and let $\omega_{\Sigma}$ be a volume form on $\Sigma$, which is compatible with the induced metric. If $\Sigma$ is not orientable, we can only choose $\omega_{\Sigma}$ locally. Then \eqref{eq:pullback_omega} implies that there exist a (local) function $\Gamma: \Sigma \to \R$ such that
\begin{equation} \label{eq:def_Gamma}
\phi_1^* \omega_{\H^2} = \phi_2^* \omega_{\H^2} = \Gamma \omega_\Sigma.
\end{equation} 

The following two theorems show some important relations between the function $\Gamma$ and the geometry of the Lagrangian surface $\Sigma$. We first prove a lemma expressing the function $\Gamma$ in terms of an orthonormal frame on $\Sigma$.

\begin{lemma} \label{lem:expressions_Gamma}
Let $\Phi=(\phi_1,\phi_2): \Sigma \to \HH$ be a Lagrangian surface and define the function~$\Gamma$ as in \eqref{eq:def_Gamma}. Let $(e_1,e_2)$ be an orthonormal frame on $\Sigma$ such that $\omega_{\Sigma}(e_1,e_2)=1$. Then
\begin{itemize}
\item[(a)] $\Gamma = \langle \phi_{j\ast}e_1 \boxtimes \phi_{j\ast}e_2 , \phi_j \rangle$,
\item[(b)] $\phi_{j\ast}e_1 \boxtimes \phi_{j\ast}e_2 = -\Gamma \phi_j$,
\item[(c)] $\Gamma^2 = -\langle \phi_{j\ast}e_1 \boxtimes \phi_{j\ast}e_2 , \phi_{j\ast}e_1 \boxtimes \phi_{j\ast}e_2 \rangle$
\end{itemize}
for $j \in \{1,2\}$, where $\langle \cdot\,,\cdot \rangle$ is the metric on $\R^3_1$.
\end{lemma}

\begin{proof}
From \eqref{eq:def_Gamma} and \eqref{eq:J_omega_H2}, we obtain $\Gamma = \Gamma \omega_{\Sigma}(e_1,e_2) = (\phi_j^{\ast}\omega_{\H^2})(e_1,e_2) = \omega_{\H^2}(\phi_{j\ast}e_1,\phi_{j\ast}e_2) = \langle J_{\H^2}\phi_{j\ast}e_1 , \phi_{j\ast}e_2 \rangle = \langle \phi_j \boxtimes \phi_{j\ast}e_1 , \phi_{j\ast}e_2 \rangle$ for $j \in \{1,2\}$. To finish the proof of (a), it suffices to remark that $\langle a \boxtimes b, c \rangle = \langle b \boxtimes c, a \rangle$ for all $a,b,c \in \R^3_1$.

To prove (b), we choose $j \in \{1,2\}$ and distinguish two cases. First, if $\phi_{j\ast}e_1 \boxtimes \phi_{j\ast}e_2 = 0$, then (a) implies that $\Gamma = 0$ and (b) reduces to $0=0$. Next, assume that $\phi_{j\ast}e_1 \boxtimes \phi_{j\ast}e_2 \neq 0$. Then $\phi_{j\ast}e_1$ and $\phi_{j\ast}e_2$ span the tangent plane to $\H^2$, which means that $\phi_{j\ast}e_1 \boxtimes \phi_{j\ast}e_2$ is a multiple of $\phi_j$. Since $\langle \phi_j,\phi_j \rangle = -1$, we have $\phi_{j\ast}e_1 \boxtimes \phi_{j\ast}e_2 = -\langle \phi_{j\ast}e_1 \boxtimes \phi_{j\ast}e_2 , \phi_j \rangle \phi_j$, which, by (a), is equivalent to (b). 

Finally, (c) is a direct consequence of (b).
\end{proof}

\begin{theorem} [Gauss equation] \label{theo:Gauss_equation}
Let $\Phi=(\phi_1,\phi_2): \Sigma \to \HH$ be a Lagrangian surface and define the function~$\Gamma$ as in \eqref{eq:def_Gamma}. Let $H$ be the mean curvature vector field of the immersion, $h$ its second fundamental form and $\|h\|^2$ the squared norm of the second fundamental form. Then the Gaussian curvature of $\Sigma$ is given by
\begin{equation} \label{eq:Gauss_equation}
K = 2\|H\|^2 - \frac{1}{2}\|h\|^2 - 2 \Gamma^2.
\end{equation}
\end{theorem}	

\begin{proof}
Choose $p \in \Sigma$ and let $(e_1,e_2)$ be an orthonormal basis of $T_p\Sigma$ such that $\omega_{\Sigma}(e_1,e_2)=1$. From the general statement of the Gauss equation, we have
\begin{equation} \label{eq:Gauss_equation2}
K = \langle R^{\HH}(\Phi_{\ast}e_1,\Phi_{\ast}e_2)\Phi_{\ast}e_2,\Phi_{\ast}e_1 \rangle + \langle h(e_1,e_1),h(e_2,e_2) \rangle - \langle h(e_1,e_2),h(e_1,e_2) \rangle.
\end{equation}
Since $H = (1/2)(h(e_1,e_1)+h(e_2,e_2))$ and $\|h\|^2=\|h(e_1,e_1)\|^2 + 2\|h(e_1,e_2)\|^2 + \|h(e_2,e_2\|^2$, we have $2\|H\|^2 - (1/2)\|h\|^2 = \langle h(e_1,e_1),h(e_2,e_2) \rangle - \langle h(e_1,e_2),h(e_1,e_2) \rangle$. By comparing \eqref{eq:Gauss_equation} and \eqref{eq:Gauss_equation2}, it is hence sufficient to prove that $\langle R^{\HH}(\Phi_{\ast}e_1,\Phi_{\ast}e_2)\Phi_{\ast}e_2,\Phi_{\ast}e_1 \rangle = -2\Gamma^2$.

In the following computation, we use \eqref{eq:curvature_tensor_HH} in the first step, the third property in \eqref{eq:properties_cross_product} in the second step and Lemma \ref{lem:expressions_Gamma} (c) in the third step:
\begin{align*}
\langle & R^{\HH}(\Phi_{\ast}e_1,\Phi_{\ast}e_2)\Phi_{\ast}e_2,\Phi_{\ast}e_1 \rangle \\
&= \langle \phi_{1\ast}e_1,\phi_{1\ast}e_2 \rangle^2 \!-\! \langle \phi_{1\ast}e_1,\phi_{1\ast}e_1 \rangle \langle \phi_{1\ast}e_2,\phi_{1\ast}e_2 \rangle
\!+\! \langle \phi_{2\ast}e_1,\phi_{2\ast}e_2 \rangle^2 \!-\! \langle \phi_{2\ast}e_1,\phi_{2\ast}e_1 \rangle \langle \phi_{2\ast}e_2,\phi_{2\ast}e_2 \rangle \\
&= \langle \phi_{1\ast}e_1 \boxtimes \phi_{1\ast}e_2 , \phi_{1\ast}e_1 \boxtimes \phi_{1\ast}e_2 \rangle + \langle \phi_{2\ast}e_1 \boxtimes \phi_{2\ast}e_2 , \phi_{2\ast}e_1 \boxtimes \phi_{2\ast}e_2 \rangle \\
&= -2\Gamma^2.
\end{align*}
\end{proof}

\begin{theorem} \label{theo:characterization_by_Gamma}
Let $\Phi=(\phi_1,\phi_2): \Sigma \to \HH$ be a Lagrangian surface and define the function~$\Gamma$ as in \eqref{eq:def_Gamma}. Then $0 \leq \Gamma^2 \leq 1/4$ and
\begin{itemize}
\item[(a)] $\Gamma^2 \equiv 0$ if and only if the image of $\Phi$ is a product of curves, as described in Example \ref{ex:product_of_curves},
\item[(b)] $\Gamma^2 \equiv 1/4$ if and only if, up to a holomorphic or anti-holomorphic isometry of $\HH$, the image of $\Phi$ is an open part of the diagonal, as described in Example \ref{ex:diagonal}.
\end{itemize}
\end{theorem}

\begin{proof}
Let $(e_1,e_2)$ be a (local) orthonormal frame on $\Sigma$ such that $\omega_{\Sigma}(e_1,e_2)=1$. From Lemma \ref{lem:expressions_Gamma} (c) and the third property in \eqref{eq:properties_cross_product}, we obtain $\Gamma^2 = -\langle \phi_{j\ast}e_1 \boxtimes \phi_{j\ast}e_2 , \phi_{j\ast}e_1 \boxtimes \phi_{j\ast}e_2 \rangle = \langle \phi_{j\ast}e_1,\phi_{j\ast}e_1 \rangle \langle \phi_{j\ast}e_2,\phi_{j\ast}e_2 \rangle - \langle \phi_{j\ast}e_1,\phi_{j\ast}e_2 \rangle^2$ for $j \in \{1,2\}$, which we can rewrite as
$$ \Gamma^2 = \frac 14 ( \|\phi_{j\ast}e_1\|^2 \!+\! \| \phi_{j\ast}e_2\|^2 )^2 - \frac 14 ( \|\phi_{j\ast}e_1\|^2 \!-\! \| \phi_{j\ast}e_2\|^2 )^2 - \langle \phi_{j\ast}e_1,\phi_{j\ast}e_2 \rangle^2. $$
Since $\Phi$ is Lagrangian, Lemma \ref{lem:Lagrangian_plane} implies that $\|\phi_{j\ast}e_1\|^2 \!+\! \| \phi_{j\ast}e_2\|^2=1$ and hence we obtain
\begin{equation} \label{eq:inequality_Gamma}
\Gamma^2 = \frac 14 - \frac 14 ( \|\phi_{j\ast}e_1\|^2 \!-\! \| \phi_{j\ast}e_2\|^2 )^2 - \langle \phi_{j\ast}e_1,\phi_{j\ast}e_2 \rangle^2 \leq \frac 14.
\end{equation}

In order to prove statement (a), we remark that it follows from Lemma \ref{lem:expressions_Gamma} (c) that $\Gamma \equiv 0$ if and only if $\phi_{1\ast}e_1 \boxtimes \phi_{1\ast}e_2 = \phi_{2\ast}e_1 \boxtimes \phi_{2\ast}e_2 = 0$, which is equivalent to the statement that both $\phi_{1\ast}$ and $\phi_{2\ast}$ are not of full rank. Since $\Phi=(\phi_1,\phi_2)$ is an immersion, and hence of full rank, we obtain that the images of $\phi_1$ and $\phi_2$ are curves in $\H^2$. Conversely, it is easy to see from Lemma \ref{lem:expressions_Gamma} that a product of curves in $\H^2$ satisfies $\Gamma^2 \equiv 0$.

Finally, $\Gamma^2=1/4$ is equivalent to $\|\phi_{j\ast}e_1\| = \| \phi_{j\ast}e_2\| = 1/\sqrt 2$ and $\langle \phi_{j\ast}e_1,\phi_{j\ast}e_2 \rangle = 0$ by \eqref{eq:inequality_Gamma} and the fact that $\|\phi_{j\ast}e_1\|^2 + \| \phi_{j\ast}e_2\|^2 = 1$. This means that both $\phi_1$ and $\phi_2$ are conformal maps from $\Sigma$ to $\H^2$ with factor $2$: the metric of $\Sigma$ is given by $g_{\Sigma} = 2 \phi_j^{\ast}\langle \cdot\,,\cdot \rangle$, where $\langle \cdot\,,\cdot \rangle$ is the metric of $\H^2$. Hence, $\Sigma$ is isometric to an open part of $\H^2(-1/2)$ and we are dealing with maps $\phi_1,\phi_2: \Sigma \subseteq \H^2(-1/2) \to \H^2$ such that $\sqrt 2 \phi_1, \sqrt 2 \phi_2: \Sigma \subseteq \H^2(-1/2) \to \H^2(-1/2)$ are isometries. This implies that there exist $A_1,A_2 \in \mathrm{O}^+(1,2)$ such that $A_1 \circ \phi_1 = \mathrm{id}/\sqrt 2$ and $A_2 \circ \phi_2 = \mathrm{id}/\sqrt 2$. Since $\phi_1^{\ast}\omega_{\H^2} = \phi_2^{\ast}\omega_{\H^2}$, we have $\det A_1 = \det A_2$ and we conclude that $\Phi=(\phi_1,\phi_2)$ is congruent to the diagonal immersion $(\mathrm{id},\mathrm{id})/\sqrt 2$ up to the (anti-)holomorphic isometry $\HH \to \HH: (x_1,x_2) \mapsto (A_1x_1,A_2x_2)$.
\end{proof}

\section{Lagrangian surfaces of $\HH$ with parallel second fundamental form}

A submanifold is said to have parallel second fundamental form if $\overline{\nabla} h=0$, where $\overline{\nabla}$ is the connection of Van der Waerden-Bortolotti. Explicitly, we have
$$ (\overline{\nabla}h)(X,Y,Z) = \nabla^{\perp}_Xh(Y,Z) - h(\nabla_XY,Z) - h(Y,\nabla_XZ) $$
for all vector fields $X$, $Y$ and $Z$ tangent to the submanifold, where $\nabla^{\perp}$ is the normal connection of the immersion and $\nabla$ is the Levi-Civita connection of the submanifold. A special case of submanifolds with parallel second fundamental form are totally geodesic submanifolds. 

In this section we classify Lagrangian surfaces of $\HH$ with parallel second fundamental form and in particular totally geodesic Lagrangian surfaces of $\HH$.

\begin{theorem} \label{theo:parallel_sff}
A Lagrangian surface in $\HH$ has parallel second fundamental form if and only if it is one of the following:
\begin{itemize}
\item[(a)] a product of curves of constant geodesic curvature in $\H^2$, as described in  Example \ref{ex:product_of_curves};
\item[(b)] a surface which is, up to a holomorphic or anti-holomorphic isometry of $\HH$, equal to an open part of the diagonal surface described in Example \ref{ex:diagonal}.
\end{itemize}
\end{theorem}

\begin{proof}
It follows from \eqref{eq:sff_product} that a product of curves has parallel second fundamental form if and only if both component curves have constant geodesic curvature in $\H^2$. Moreover, we already know that the diagonal surface is totally geodesic and hence has parallel second fundamental form. It is therefore sufficient to prove that a Lagrangian immersion $\Phi=(\phi_1,\phi_2): \Sigma \to \HH$ with parallel second fundamental form is either a product of curves or (anti-)holomorphically congruent to an open part of the diagonal surface.

From the equation of Codazzi, we have
\begin{equation} \label{eq:Codazzi}
(\overline{\nabla}h)(X,Y,Z) - (\overline{\nabla}h)(Y,X,Z) = (R^{\HH}(\Phi_{\ast}X,\Phi_{\ast}Y)\Phi_{\ast}Z)^{\perp}
\end{equation}
for all vector fields $X$, $Y$ and $Z$ tangent to $\Sigma$, where the superscript $\perp$ denotes the normal component. Let $(e_1,e_2)$ be a (local) orthonormal frame on $\Sigma$ such that $\omega_{\Sigma}(e_1,e_2)=1$. We will now compute $\langle R^{\HH}(\Phi_{\ast}e_1,\Phi_{\ast}e_2)\Phi_{\ast}e_1, J\Phi_{\ast}e_1 \rangle$, $\langle R^{\HH}(\Phi_{\ast}e_1,\Phi_{\ast}e_2)\Phi_{\ast}e_1, J\Phi_{\ast}e_2 \rangle$ and $\langle R^{\HH}(\Phi_{\ast}e_1,\Phi_{\ast}e_2)\Phi_{\ast}e_2, J\Phi_{\ast}e_2 \rangle$. (Note that, since $\HH$ is a K\"ahler manifold, we know that $\langle R^{\HH}(\Phi_{\ast}e_1,\Phi_{\ast}e_2)\Phi_{\ast}e_2, J\Phi_{\ast}e_1 \rangle$ is equal to $\langle R^{\HH}(\Phi_{\ast}e_1,\Phi_{\ast}e_2)\Phi_{\ast}e_1, J\Phi_{\ast}e_2 \rangle$, so we don't compute the former.) Since $\Phi$ is Lagrangian and $\overline{\nabla}h=0$, equation \eqref{eq:Codazzi} implies that all three inner products are zero.

A straightforward computation, using \eqref{eq:curvature_tensor_HH} and Lemma \ref{lem:expressions_Gamma} (a) yields that
\begin{equation} \label{eq:Codazzi2}
\begin{aligned}
& 0 = \langle R^{\HH}(\Phi_{\ast}e_1,\Phi_{\ast}e_2)\Phi_{\ast}e_1, J\Phi_{\ast}e_1 \rangle = \Gamma (\|\phi_{1\ast}e_1\|^2-\|\phi_{2\ast}e_1\|^2), \\
& 0 = \langle R^{\HH}(\Phi_{\ast}e_1,\Phi_{\ast}e_2)\Phi_{\ast}e_1, J\Phi_{\ast}e_2 \rangle = \Gamma (\langle\phi_{1\ast}e_1,\phi_{1\ast}e_2\rangle - \langle\phi_{2\ast}e_1,\phi_{2\ast}e_2\rangle), \\
& 0 = \langle R^{\HH}(\Phi_{\ast}e_1,\Phi_{\ast}e_2)\Phi_{\ast}e_2, J\Phi_{\ast}e_2 \rangle = \Gamma (\|\phi_{1\ast}e_2\|^2-\|\phi_{2\ast}e_2\|^2).
\end{aligned}
\end{equation}
We now consider two cases. First, if $\Gamma=0$, it follows from Theorem \ref{theo:characterization_by_Gamma} (a), that $\Phi$ is a product of curves. If $\Gamma \neq 0$, then \eqref{eq:Codazzi2} implies that $\|\phi_{1\ast}e_1\|^2 = \|\phi_{2\ast}e_1\|^2$, $\langle\phi_{1\ast}e_1,\phi_{1\ast}e_2\rangle = \langle\phi_{2\ast}e_1,\phi_{2\ast}e_2\rangle$ and $\|\phi_{1\ast}e_2\|^2 = \|\phi_{2\ast}e_2\|^2$. Since we also have $\|\phi_{1\ast}e_1\|^2 + \|\phi_{2\ast}e_1\|^2 = \|e_1\|^2=1$, $\langle\phi_{1\ast}e_1,\phi_{1\ast}e_2\rangle + \langle\phi_{2\ast}e_1,\phi_{2\ast}e_2\rangle = \langle e_1,e_2 \rangle = 0$ and $\|\phi_{1\ast}e_2\|^2 + \|\phi_{2\ast}e_2\|^2 = \|e_2\|^2 = 1$, we obtain that $\|\phi_{1\ast}e_1\| = \|\phi_{2\ast}e_1\| = \|\phi_{1\ast}e_2\| = \|\phi_{2\ast}e_2\| = 1/\sqrt 2$ and $\langle\phi_{1\ast}e_1,\phi_{1\ast}e_2\rangle = \langle\phi_{2\ast}e_1,\phi_{2\ast}e_2\rangle = 0$. It then follows from \eqref{eq:inequality_Gamma} that $\Gamma^2 = 1/4$ and Theorem \ref{theo:characterization_by_Gamma} implies that the surface at hand is (anti-)holomorphically congruent to an open part of the diagonal surface.
\end{proof}

\begin{theorem} \label{theo:totally_geodesic}
A Lagrangian surface in $\HH$ is totally geodesic if and only if it is one of the following:
\begin{itemize}
\item[(a)] a product of geodesics in $\H^2$, as described in  Example \ref{ex:product_of_curves};
\item[(b)] a surface which is, up to a holomorphic or anti-holomorphic isometry of $\HH$, equal to an open part of the diagonal surface described in Example \ref{ex:diagonal}.
\end{itemize}
\end{theorem}

\begin{proof}
Since totally geodesic surfaces have parallel second fundamental form, we can use Theorem \ref{theo:parallel_sff}. To finish the proof, it suffices to remark that \eqref{eq:sff_product} implies that a product of curves is totally geodesic if and only if both curves are geodesics.
\end{proof}

\begin{remark}
It is well-known that totally umbilical Lagrangian submanifolds of K\"ahler manifolds are totally geodesic. Hence, Theorem \ref{theo:totally_geodesic} also provides a full classification of totally umbilical Lagrangian surfaces in $\HH$.
\end{remark}

\section{Minimal Lagrangian surfaces of $\HH$}

In this section, we will prove two theorems about minimal Lagrangian surfaces in $\HH$. Theorem \ref{theo:minimal1} states that if such a surface has constant Gaussian curvature, it must be totally geodesic and hence be one of the two surfaces given in Theorem \ref{theo:totally_geodesic}. Theorem \ref{theo:minimal2} states that if such a surface is complete and the function $\Gamma^2$, defined in \eqref{eq:def_Gamma}, is bounded away from $1/4$, then the surface must be a product of geodesics.

We start by proving a lemma about minimal surfaces in $\HH$. Recall that a minimal surface $\Sigma$ with second fundamental form $h$ is said to be \textit{superminimal} if, at every point $p \in \Sigma$, the length of $h(e,e)$ does not depend on the unit vector $e \in T_p\Sigma$. Equivalently: a minimal surface is superminimal if all its ellipses of curvature are circles.

\begin{lemma}
A minimal Lagrangian surface $\Sigma$ of $\HH$ is always superminimal and its Gaussian curvature at a point $p \in \Sigma$ is given by
\begin{equation}\label{eq:Gauss_equation4}
K(p) = -2(\|h(e,e)\|^2+\Gamma(p)^2),
\end{equation} 
where $h$ is the second fundamental form, $e$ is any unit vector tangent to $\Sigma$ at $p$ and $\Gamma$ is the function defined by \eqref{eq:def_Gamma}.
\end{lemma}

\begin{proof}
The fact that a minimal surface is superminimal holds for any Lagrangian surface in a K\"ahler manifold, but we will sketch the proof for completeness. Let $\Phi: \Sigma \to \HH$ be a minimal Lagrangian immersion and fix a point $p \in \Sigma$ and an orthonormal basis $\{e_1,e_2\}$ of $T_p\Sigma$. A straightforward computation, using $h(e_1,e_1)+h(e_2,e_2)=0$, yields
\begin{multline*} 
\| h((\cos\theta)e_1 \!+\! (\sin\theta)e_2 , \, (\cos\theta)e_1 \!+\! (\sin\theta)e_2) \|^2 = \|h(e_1,e_1)\|^2 + \|h(e_1,e_2)\|^2 \\ + \frac 12 \cos(4\theta) (\|h(e_1,e_1)\|^2 - \|h(e_1,e_2)\|^2) + \sin(4\theta) \langle h(e_1,e_1),h(e_1,e_2)\rangle
\end{multline*}
for every $\theta \in \R$. Superminimality means that the above expression is independent of $\theta$, which is equivalent to $\|h(e_1,e_1)\|^2 = \|h(e_1,e_2)\|^2$ and $\langle h(e_1,e_1),h(e_1,e_2)\rangle = 0$. These two equalities can be verified using minimality and the fact that the cubic form $(X,Y,Z) \mapsto \langle h(X,Y),J\Phi_{\ast}Z \rangle$ is symmetric in its three arguments, a property that holds for any Lagrangian submanifold of a K\"ahler manifold. 

In this situation, the Gauss equation \eqref{eq:Gauss_equation} reduces to $K(p) = -2\|h(e_1,e_1)\|^2-2\Gamma(p)^2$. By superminimality, one may replace the vector $e_1$ by any unit vector in $T_p\Sigma$, which proves  \eqref{eq:Gauss_equation4}.
\end{proof}

In the rest of this section, we will use a complex coordinate on a surface. Let $(\Sigma,g_{\Sigma})$ be a Riemannian surface and let $u$ and $v$ be local isothermal coordinates on the surface, say $g_{\Sigma}(\partial_u,\partial_v) = 0$ and $g_{\Sigma}(\partial_u,\partial_u) = g_{\Sigma}(\partial_v,\partial_v) = e^{2f(u,v)}$ for some real-valued function $f$. We also assume that these coordinates are compatible with the chosen (local) orientation of $\Sigma$ in the sense that the orthonormal vectors $e_1=e^{-f}\partial_u$ and $e_2=e^{-f}\partial_v$ satisfy $\omega_{\Sigma}(e_1,e_2)=1$. If we put $z=u+iv$ and $\bar z=u-iv$, we have $\partial_z = (\partial_u-i\partial_v)/2$ and $\partial_{\bar z}=(\partial_u+i\partial_v)/2$. If $\Phi: \Sigma \to \HH$ is an isometric immersion and $\langle \cdot\,,\cdot \rangle$ is the complex linear extension of the metric on $\HH$, then 
\begin{equation} \label{eq:isothermal}
\langle \Phi_z,\Phi_z \rangle = \langle \Phi_{\bar z},\Phi_{\bar z} \rangle = 0, \qquad \langle \Phi_z,\Phi_{\bar{z}} \rangle = \frac{e^{2f}}{2}.
\end{equation}

Now assume that $\Phi$ is Lagrangian. We want to use Bochner's formula for the function $\Gamma$, defined in \eqref{eq:def_Gamma}. Therefore, we will compute the Laplacian of $\Gamma$ and the squared norm of the gradient of $\Gamma$. We first prove the following lemma.

\begin{lemma} \label{lem:minimal1}
Let $\Phi=(\phi_1,\phi_2): \Sigma \to \HH$ be a minimal Lagrangian immersion and let $\Gamma$ be the defined by \eqref{eq:def_Gamma}. If $z=u+iv$ is a complex coordinate on $\Sigma$ such that \eqref{eq:isothermal} holds, then
\begin{align}
& \Phi_{z\bar{z}} = \frac{1}{4}e^{2f}\Phi, \label{eq:Phi_z_zbar} \\
& (J\Phi_{\bar{z}})_z = -\frac{i}{2} \Gamma e^{2f} \hat{\Phi}, \label{eq:JPhi_zbar_z} \\
& \Phi_{zz} = 2f_z\Phi_z + 2e^{-2f} \langle \Phi_{zz},J\Phi_z \rangle J\Phi_{\bar{z}} + \frac{1}{2} \langle \Phi_z,\hat{\Phi}_z \rangle \hat{\Phi}, \label{eq:Phi_zz}
\end{align}
where $J$ is the complex linear extension of the complex structure on $\HH$ and $\hat{\Phi} = (\phi_1,-\phi_2)$.
\end{lemma}

\begin{proof}
Since $\Phi$ is minimal in $\HH$, the mean curvature vector field of $\Phi$, seen as an immersion of $\Sigma$ into $\R^3_1 \times \R^3_1 \cong \R^6_2$, coincides with the mean curvature vector field of $\H^2 \times \H^2$ as a submanifold of $\R^6_2$, restricted to the image of $\Phi$. Since the second fundamental form of the latter submanifold is $\tilde h_{(x_1,x_2)}((v_1,v_2),(w_1,w_2)) = (\langle v_1,w_1 \rangle x_1, \langle v_2,w_2 \rangle x_2)$ for all $x_1, x_2 \in \H^2$ and all $v_1,w_1 \in T_{x_1}\H^2$ and $v_2,w_2 \in T_{x_2}\H^2$, the mean curvature vector field of $\Phi: \Sigma \to \R^6_2$ is
\begin{equation} \label{eq:mean_curvature1}
H = \frac 12 (\phi_1,\phi_2) = \frac 12 \Phi.
\end{equation}
On the other hand, this mean curvature vector field can be expressed using the Laplacian of $\Sigma$ as
\begin{equation} \label{eq:mean_curvature2}
H = \frac 12 \Delta\Phi = \frac 12 e^{-2f}(\Phi_{uu}+\Phi_{vv}) = 2 e^{-2f} \Phi_{z\bar z}.
\end{equation}
By comparing \eqref{eq:mean_curvature1} and \eqref{eq:mean_curvature2}, we obtain \eqref{eq:Phi_z_zbar}.

By definition of $J$, we have that 
\begin{equation*}
(J\Phi_{\bar z})_z = (\phi_1 \boxtimes \phi_{1\bar z}, -\phi_2 \boxtimes \phi_{2\bar z})_z = (\phi_{1z} \boxtimes \phi_{1\bar z}, -\phi_{2z} \boxtimes \phi_{2\bar z}) + (\phi_1 \boxtimes \phi_{1z\bar z}, -\phi_2 \boxtimes \phi_{2z\bar z}). 
\end{equation*}
Equation \eqref{eq:Phi_z_zbar} implies that the second term on the right hand side vanishes and rewriting the first term in real coordinates yields
\begin{equation*}
(J\Phi_{\bar z})_z = (\phi_{1z} \boxtimes \phi_{1\bar z}, -\phi_{2z} \boxtimes \phi_{2\bar z}) = \frac i2 (\phi_{1u} \boxtimes \phi_{1v}, -\phi_{2u} \boxtimes \phi_{2v}).
\end{equation*} 
Now, let $(e_1,e_2)$ be the orthonormal frame given by $e_1=e^{-f}\partial_u$ and $e_2=e^{-f}\partial_v$. We know from Lemma \ref{lem:expressions_Gamma} (b) that $\phi_{j\ast}e_1 \boxtimes \phi_{j\ast}e_2 = -\Gamma \phi_j$ for $j \in \{1,2\}$ and therefore
\begin{equation*}
(J\Phi_{\bar z})_z = \frac i2 e^{2f} (\phi_{1\ast}e_1 \boxtimes \phi_{1\ast}e_2, -\phi_{2\ast}e_1 \boxtimes \phi_{2\ast}e_2) = \frac i2 e^{2f} (-\Gamma \phi_1,\Gamma \phi_2) = -\frac{i}{2} \Gamma e^{2f} (\phi_1,-\phi_2),
\end{equation*} 
which finishes the proof of \eqref{eq:JPhi_zbar_z}.

Finally, to prove \eqref{eq:Phi_zz}, we express $\Phi_{zz}$ in the basis $\{ \Phi_z, \Phi_{\bar z}, J\Phi_z, J\Phi_{\bar z}, \Phi, \hat{\Phi} \}$. Since the only non-zero inner products of these vectors are 
\begin{equation} \label{eq:scalar_products}
\langle \Phi_z,\Phi_{\bar{z}} \rangle = \langle J\Phi_z,J\Phi_{\bar{z}} \rangle = \frac{e^{2f}}{2}, \qquad \langle \Phi,\Phi \rangle = \langle \hat{\Phi},\hat{\Phi} \rangle = -2,
\end{equation}
the components are determined by scalar products of $\Phi_{zz}$ with suitable vectors as follows:
\begin{multline} \label{eq:Phi_zz1}
\Phi_{zz} = 2e^{-2f} \langle \Phi_{zz},\Phi_{\bar{z}} \rangle \Phi_z 
+ 2e^{-2f} \langle \Phi_{zz},\Phi_z \rangle \Phi_{\bar{z}} \\
+ 2e^{-2f} \langle \Phi_{zz},J\Phi_{\bar{z}} \rangle J\Phi_z
+ 2e^{-2f} \langle \Phi_{zz},J\Phi_z \rangle J\Phi_{\bar{z}}
- \frac{1}{2} \langle \Phi_{zz},\Phi \rangle \Phi
- \frac{1}{2} \langle \Phi_{zz},\hat{\Phi} \rangle \hat{\Phi}.
\end{multline}
When comparing \eqref{eq:Phi_zz1} to \eqref{eq:Phi_zz}, we see that it is now sufficient to prove the following five equalities: $\langle \Phi_{zz},\Phi_{\bar z} \rangle = e^{2f}f_z$, $\langle \Phi_{zz},\Phi_z \rangle = 0$, $\langle \Phi_{zz},J\Phi_{\bar z} \rangle = 0$, $\langle \Phi_{zz},\Phi \rangle = 0$ and $\langle \Phi_{zz},\hat{\Phi} \rangle = - \langle \Phi_z,\Phi_{\bar z} \rangle$. These follow from straightforward computations using that \eqref{eq:scalar_products} gives the only non-zero scalar products of the six basis vectors, \eqref{eq:Phi_z_zbar} and \eqref{eq:JPhi_zbar_z}.
\end{proof}

We can now compute the squared norm of the gradient and the Laplacian of the function $\Gamma$ on a minimal Lagrangian surface in $\HH$.

\begin{lemma}\label{lem:Gamma_isoparametric}
Let $\Phi: \Sigma \to \HH$ be a minimal Lagrangian immersion and let $\Gamma$ be the function defined by \eqref{eq:def_Gamma}. Then the squared norm of the gradient and the Laplacian of $\Gamma$ are
\begin{equation*}
\|\nabla \Gamma\|^2 = \frac 12 (4\Gamma^2-1)(2\Gamma^2+K), \qquad \Delta \Gamma = \Gamma(4\Gamma^2+4K+1),
\end{equation*}
where $K$ is the Gaussian curvature of $\Sigma$. In particular, if $K$ is constant, the function $\Gamma$ is isoparametric.
\end{lemma}

\begin{proof}
Using a complex coordinate $z=u+iv$ satisfying \eqref{eq:isothermal}, we have to compute 
\begin{align}
& \|\nabla\Gamma\|^2 = e^{-2f}(\Gamma_u^2+\Gamma_v^2) = 4e^{-2f} |\Gamma_z|^2, \label{eq:gradient}\\ 
& \Delta \Gamma = e^{-2f}(\Gamma_{uu}+\Gamma_{vv}) = 4e^{-2f} \Gamma_{z \bar z}. \label{eq:Laplacian}
\end{align}

First, note that, using \eqref{eq:JPhi_zbar_z}, we have
\begin{equation} \label{eq:Gamma1}
\langle J\Phi_{\bar{z}},\hat{\Phi}_z \rangle = \langle J\Phi_{\bar{z}},\hat{\Phi} \rangle_z - \langle (J\Phi_{\bar{z}})_z,\hat{\Phi} \rangle = -i \Gamma e^{2f}.
\end{equation}
Deriving this equality with respect to $z$ gives $\langle (J\Phi_{\bar{z}})_z,\hat{\Phi}_z \rangle + \langle J\Phi_{\bar{z}},\hat{\Phi}_{zz} \rangle = -i \Gamma_z e^{2f} - 2i\Gamma e^{2f}f_z$, which, by using \eqref{eq:JPhi_zbar_z}, \eqref{eq:Phi_zz} and \eqref{eq:Gamma1}, can be simplified to
\begin{equation} \label{eq:Gamma2}
\Gamma_z = 2i e^{-4f} \langle \Phi_{zz},J\Phi_z \rangle \langle \Phi_{\bar z},\hat{\Phi}_{\bar z} \rangle.
\end{equation}
	
We now investigate the two scalar products appearing in \eqref{eq:Gamma2}. In the computations below, we use the orthonormal frame given by $e_1=e^{-f}\partial_u$ and $e_2=e^{-f}\partial_v$. First, remark that
\begin{equation*}
\begin{aligned}
\langle \Phi_{zz},J\Phi_z \rangle & = \frac 18 ( \langle \Phi_{uu},J\Phi_u \rangle - \langle \Phi_{vv},J\Phi_u \rangle - 2 \langle \Phi_{uv},J\Phi_v \rangle ) \\
& \qquad - \frac i8 ( 2 \langle \Phi_{uv},J\Phi_u \rangle + \langle \Phi_{uu},J\Phi_v \rangle - \langle \Phi_{vv},J\Phi_v \rangle ) \\
& = \frac 18 e^{3f} ( \langle h(e_1,e_1),J\Phi_{\ast}e_1 \rangle - \langle h(e_2,e_2),J\Phi_{\ast}e_1 \rangle - 2 \langle h(e_1,e_2),J\Phi_{\ast}e_2 \rangle ) \\
& \qquad - \frac i8 e^{3f} ( 2 \langle h(e_1,e_2),J\Phi_{\ast}e_1 \rangle + \langle h(e_1,e_1),J\Phi_{\ast}e_2 \rangle - \langle h(e_2,e_2),J\Phi_{\ast}e_2 \rangle ),
\end{aligned}
\end{equation*}	
which, by minimality and symmetry of the cubic form, can be reduced to	
\begin{equation} \label{eq:Gamma3}
\langle \Phi_{zz},J\Phi_z \rangle = \frac 12 e^{3f} ( \langle h(e_1,e_1),J\Phi_{\ast}e_1 \rangle - i \langle h(e_1,e_1),J\Phi_{\ast}e_2 \rangle ).
\end{equation}	
Combining \eqref{eq:Gamma3} and \eqref{eq:Gauss_equation4} yields
\begin{equation} \label{eq:Gamma4}
| \langle \Phi_{zz},J\Phi_z \rangle |^2 = \frac 14 e^{6f} \|h(e_1,e_1)\|^2 = - \frac 18 e^{6f} (2\Gamma^2 + K).
\end{equation}

Next, denoting $\Phi=(\phi_1,\phi_2)$ as usual, we have
\begin{equation*}
\begin{aligned}
\langle \Phi_{\bar z},\hat{\Phi}_{\bar z} \rangle & = \frac 14 ( \|\phi_{1u}\|^2 \!-\! \|\phi_{2u}\|^2 \!-\! \|\phi_{1v}\|^2 \!+\! \|\phi_{2v}\|^2 ) + \frac i2 ( \langle \phi_{1u},\phi_{1v} \rangle \!-\! \langle \phi_{2u},\phi_{2v} \rangle ) \\
& = \frac 14 e^{2\omega} ( \|\phi_{1\ast}e_1\|^2 \!-\! \|\phi_{2\ast}e_1\|^2 \!-\! \|\phi_{1\ast}e_2\|^2 \!+\! \|\phi_{2\ast}e_2\|^2 ) + \frac i2 ( \langle \phi_{1\ast}e_1,\phi_{1\ast}e_2 \rangle \!-\! \langle \phi_{2\ast}e_1,\phi_{2\ast}e_2 \rangle ),
\end{aligned}
\end{equation*}	
which, by Lemma \ref{lem:Lagrangian_plane} can be simplified to
\begin{equation} \label{eq:Gamma5}
\langle \Phi_{\bar z},\hat{\Phi}_{\bar z} \rangle = \frac 12 e^{2f} ( \|\phi_{1\ast}e_1\|^2 \!-\! \|\phi_{2\ast}e_1\|^2 + i ( \langle \phi_{1\ast}e_1,\phi_{1\ast}e_2 \rangle \!-\! \langle \phi_{2\ast}e_1,\phi_{2\ast}e_2 \rangle ) ).
\end{equation}
A straightforward computation, using \eqref{eq:Gamma5}, the facts that $\|\phi_{1\ast}e_1\|^2 + \|\phi_{2\ast}e_1\|^2 = \|e_1\|^2 = 1$ and $\langle \phi_{1\ast}e_1,\phi_{1\ast}e_2 \rangle + \langle \phi_{2\ast}e_1,\phi_{2\ast}e_2 \rangle = \langle e_1,e_2 \rangle = 0$, Lemma \ref{lem:Lagrangian_plane}, the third property in \eqref{eq:properties_cross_product} and Lemma \ref{lem:expressions_Gamma} (c), yields
\begin{equation} \label{eq:Gamma6}
|\langle \Phi_{\bar z},\hat{\Phi}_{\bar z} \rangle|^2 = \frac 14 e^{4f} (1-4\Gamma^2).
\end{equation}

We can now compute the squared norm of the gradient of $\Gamma$ from \eqref{eq:gradient}, \eqref{eq:Gamma2}, \eqref{eq:Gamma4} and \eqref{eq:Gamma6} to be
\begin{equation*}
\|\nabla\Gamma\|^2 = 16e^{-10f} |\langle \Phi_{zz},J\Phi_z \rangle|^2 |\langle \Phi_{\bar z},\hat{\Phi}_{\bar z} \rangle|^2 = \frac 12 (4 \Gamma^2-1)(2\Gamma^2+K).
\end{equation*}

To find an expression for the Laplacian of $\Gamma$, we need to derive \eqref{eq:Gamma2} with respect to $\bar z$. Using \eqref{eq:Phi_z_zbar}, \eqref{eq:JPhi_zbar_z} and \eqref{eq:Phi_zz}, we obtain
\begin{equation} \label{eq:Gamma7}
\langle \Phi_{zz},J\Phi_z \rangle_{\bar z} = -\frac i2 \Gamma e^{2f} \langle \Phi_z,\hat{\Phi}_z \rangle
\end{equation}
and using \eqref{eq:Phi_zz} and \eqref{eq:Gamma1} yields
\begin{equation} \label{eq:Gamma8}
\langle \Phi_{\bar z},\hat{\Phi}_{\bar z} \rangle_{\bar z} = 4 f_{\bar z} \langle \Phi_{\bar z},\hat{\Phi}_{\bar z} \rangle + 4i\Gamma \langle \Phi_{\bar z \bar z},J\Phi_{\bar z} \rangle.
\end{equation}
By combining \eqref{eq:Laplacian}, \eqref{eq:Gamma2}, \eqref{eq:Gamma7}, \eqref{eq:Gamma8}, \eqref{eq:Gamma4} and \eqref{eq:Gamma6}, we obtain
\begin{equation*}
\Delta\Gamma = 4e^{-2f} ( 2i e^{-4f} \langle \Phi_{zz},J\Phi_z \rangle \langle \Phi_{\bar z},\hat{\Phi}_{\bar z} \rangle )_{\bar z} = \Gamma (4\Gamma^2+4K+1).
\end{equation*}
\end{proof}

We can now prove the main theorems of this section.

\begin{theorem} \label{theo:minimal1}
Let $\Sigma$ be a minimal Lagrangian surface in $\HH$. If $\Sigma$ has constant Gaussian curvature, then the immersion is totally geodesic. Hence, it is a product of geodesics or (anti-) holomorphically congruent to an open part of the diagonal surface of Example \ref{ex:diagonal}. In particular, the constant Gaussian curvature can only be $0$ or $-1/2$.
\end{theorem}

\begin{proof}
Let $U$ be the open subset of $\Sigma$ on which $\nabla \Gamma \neq 0$. We will show that $U = \emptyset$ and hence that $\Gamma$ constant on $\Sigma$ (which we assume to be connected). For an orthonormal frame $\{e_1,e_2\}$ on $U$, we have Bochner's formula
\begin{equation} \label{eq:Bochner}
\frac{1}{2} \Delta\|\nabla\Gamma\|^2 = K \|\nabla\Gamma\|^2 + \langle \nabla\Gamma,\nabla(\Delta\Gamma) \rangle + \sum_{i=1}^{2} \|\nabla_{e_i}\nabla\Gamma\|^2.
\end{equation}
In general, if $F$ is a differentiable real-valued function in one variable, then $\nabla (F\circ\Gamma)=(F'\circ\Gamma) \nabla\Gamma$ and $\Delta (F\circ\Gamma) = (F'\circ\Gamma) \Delta\Gamma + (F''\circ\Gamma)\|\nabla\Gamma\|^2$. Therefore, using Lemma \ref{lem:Gamma_isoparametric} and the fact that $K$ is constant, we obtain 
\begin{align}
& \nabla(\Delta\Gamma) = (12\Gamma^2+4K+1) \nabla\Gamma, \label{eq:Bochner1} \\
& \frac 12 \Delta\|\nabla\Gamma\|^2 = \Gamma(8\Gamma^2+2K-1)\Delta\Gamma + (24\Gamma^2+2K-1)\|\nabla\Gamma\|^2. \label{eq:Bochner2}
\end{align}
Moreover, if we take the orthonormal frame $\{e_1,e_2\}$ on $U$ such that $e_1=\nabla\Gamma/\|\nabla\Gamma\|$, a straightforward computation using Lemma \ref{lem:Gamma_isoparametric} implies
\begin{align}
\sum_{i=1}^{2} \|\nabla_{e_i} \nabla\Gamma \|^2 &= \langle \nabla_{e_1}\nabla\Gamma, e_1 \rangle^2 + \langle \nabla_{e_1}\nabla\Gamma, e_2 \rangle^2 + \langle \nabla_{e_2}\nabla\Gamma, e_1 \rangle^2 + \langle \nabla_{e_2}\nabla\Gamma, e_2 \rangle^2 \label{eq:Bochner3} \\
&= \langle \nabla_{e_1}\nabla\Gamma, e_1 \rangle^2 + 2 \langle \nabla_{e_2}\nabla\Gamma, e_1 \rangle^2 + \langle \nabla_{e_2}\nabla\Gamma, e_2 \rangle^2 \nonumber \\
&= (e_1(\|\nabla\Gamma\|))^2 + 2 \cdot 0 + (\Delta\Gamma - e_1(\|\nabla\Gamma\|))^2 \nonumber \\
&= \Gamma^2(8\Gamma^2+2K-1)^2 + 4\Gamma^2 (2\Gamma^2-K-1)^2. \nonumber
\end{align}
Substituting \eqref{eq:Bochner1}, \eqref{eq:Bochner2} and \eqref{eq:Bochner3} into \eqref{eq:Bochner} and using Lemma \ref{lem:Gamma_isoparametric} again, yields
\begin{equation*}
\frac 12 (4\Gamma^2-1)(2(9K+4)\Gamma^2 - K(3K+2)) = 0.
\end{equation*}
There are no values of $K$ for which the polynomial on the left hand side is identically zero. Hence $\Gamma$ is constant on every connected component of $U$, which is impossible since $\nabla \Gamma \neq 0$ on $U$. We conclude that $U= \emptyset$, which means that $\nabla\Gamma$ vanishes everywhere, such that $\Gamma$ is constant on $\Sigma$. 

Again from Lemma \ref{lem:Gamma_isoparametric}, we see that there are two possibilities, namely (1) $\Gamma=0$ and $K=0$ and (2) $\Gamma^2=1/4$ and $K=-1/2$. Theorem \ref{theo:characterization_by_Gamma} finishes the proof. Note that a product of curves is only minimal if both curves are geodesics, as can be seen from \eqref{eq:sff_product}.
\end{proof}

\begin{theorem} \label{theo:minimal2}
Let $\Sigma$ be a complete minimal Lagrangian surface in $\HH$. If the function $\Gamma$, defined by $\eqref{eq:def_Gamma}$, satisfies 
$$ \Gamma^2 \leq \frac 14 - \varepsilon $$
for some positive constant $\varepsilon$, then $\Gamma=0$ and $\Sigma$ is a product of geodesics.
\end{theorem}

\begin{proof}	
The function $1-4\Gamma^2$ is positive and it follows from Lemma \ref{lem:Gamma_isoparametric} that
\begin{equation} \label{eq:laplacian_of_1-4Gamma2}
\Delta \log (1-4\Gamma^2) = 4K.
\end{equation}
If $g_{\Sigma}$ is the metric of $\Sigma$, then \eqref{eq:laplacian_of_1-4Gamma2} implies that $g_0 = \sqrt{1-4\Gamma^2}\,g_{\Sigma}$ is a flat metric on $\Sigma$. Since $g_{\Sigma}$ is complete, $g_0$ is also complete. If $\Delta_0$ is the Laplacian of $g_0$, we deduce from \eqref{eq:laplacian_of_1-4Gamma2} and \eqref{eq:Gauss_equation4} that 
\begin{equation} \label{eq:laplacian2_of_1-4Gamma2}
\Delta_0 \log (1-4\Gamma^2)=\frac{4K}{\sqrt{1-4\Gamma^2}}\leq 0.
\end{equation}
Since $1-4\Gamma^2$ is bounded away from zero, $\log(1-4\Gamma^2)$ is a bounded function on a complete flat surface and \eqref{eq:laplacian2_of_1-4Gamma2} shows that it is superharmonic. Such a function must be constant and, again from \eqref{eq:laplacian2_of_1-4Gamma2}, we obtain that $K\equiv 0$. The result now follows from Theorem \ref{theo:minimal1}.
\end{proof}

\end{document}